\documentclass{birkjour}
\usepackage {amsmath, amscd} 
\usepackage {amssymb}
\usepackage {epsfig}
\usepackage {bm}
\usepackage {indentfirst} 
\usepackage{color}

\usepackage[active]{srcltx}

\numberwithin{equation}{section}

\def\<{\langle}
\def\>{\rangle}

\def\K{{\bf K}}
\def\L{{\bf L}}

\def\O{{\bf O}}

\def\T{{\bf T}}
\def\U{{\bf U}}

\def\AA{{\mathcal A}}

\def\DD{{\mathcal D}}
\def\EE{{\mathcal E}}

\def\HH{{\mathcal H}}

\def\KK{{\mathcal K}}
\def\LL{{\mathcal L}}

\def\SS{{\mathcal S}}

\def\UU{{\mathcal U}}

\def\bbC{\mathbb{C}}

\def\bbD{\mathbb{D}}
\def\bbT{\mathbb{T}}
\def\bbH{\mathbb{H}}

\def\KKK{\mathfrak{K}}
\def\ddd{\mathfrak{d}}

\newcommand{\rank}{\mathop{\rm rank}}

\newcommand{\conv}{\mathop{\rm co}}

\newtheorem{lemma}{Lemma}[section]
\newtheorem{proposition}[lemma]{Proposition}
\newtheorem{theorem}[lemma]{Theorem}
\newtheorem{corollary}[lemma]{Corollary}
\newtheorem{conjecture}[lemma]{Conjecture}

\theoremstyle{definition}
\newtheorem{remark}[lemma]{Remark}
\newtheorem{definition}[lemma]{Definition}
\title{Numerical ranges of $C_0(N)$ contractions}
\author{Chafiq Benhida} 
\date{}

\address{UFR de Math\'ematiques, Universit\'e des Sciences et Technologies
de Lille, F-59655 Villeneuve D'Ascq Cedex, France}
\email{Chafiq.Benhida@math.univ-lille1.fr}

\author{Pamela Gorkin}

\address{Department of Mathematics, Bucknell University, Lewisburg, PA 17837, U.S.A.}
\email{pgorkin@bucknell.edu}

\author{Dan Timotin}
\address{Institute of Mathematics of the Romanian Academy, P.O. Box 1-764, Bucharest 
014700, Romania}
\email{Dan.Timotin@imar.ro}


\begin{document}

\maketitle

\begin{abstract}
A conjecture of Halmos proved by Choi and Li states that the closure of the numerical range of a contraction on a Hilbert space is the intersection of the closure of the numerical ranges of all its unitary dilations. We show that for $C_0(N)$ contractions one can restrict the intersection to a smaller family of dilations. This generalizes a finite dimensional result of Gau and Wu.
\end{abstract}

\newcommand{\mysqth}{(1-|\theta(0)|^2)^{1/2}}
\newcommand{\mysql}{(1-|\lambda|^2)^{1/2}}

\section{Introduction}

 Suppose $\HH, \HH'$ are separable Hilbert spaces; we will denote by $\LL(\HH,\HH')$ the space of bounded linear operators $T:\HH\to\HH'$ and $\LL(\HH)=\LL(\HH,\HH)$. The numerical range of an operator $T\in\LL(\HH)$  is the set $$W(T):= \{\langle Tx, x\rangle: \|x\| = 1\}.$$ Much is known about this set; for example, it is convex, in the finite-dimen\-sional case it is compact, and if $T$ is normal, the closure of $W(T)$ is the convex hull of  the spectrum of $T$. In general, however, the numerical range is difficult to compute. In this paper, we study new ways of obtaining the numerical range of a contraction $T$ from the numerical ranges of certain unitary dilations of~$T$. 

If there is a Hilbert space $\KK$ containing $\HH$ and an operator $\tilde{T} \in \LL(\KK)$ such that $T = P_{\HH}\tilde{T}|\HH$, where $P_{\HH}$ denotes the orthogonal projection onto $\HH$, the operator $T$ is said to {\it dilate} to the operator $\tilde{T}$. (We note that we are considering the so-called {\it weak dilations} here, and not power dilations treated in Sz.-Nagy dilation theory.) The operator $\tilde{T}$ is said to be {\it a dilation of} $T$; more precisely, if $\dim(\KK\ominus \HH)=k$, then $\tilde T$ is called a \emph{$k$-dilation}.

We will be interested in unitary dilations. A result of Halmos \cite[Problem 222(a)]{Halmosproblem} shows that every contraction $T$  has  unitary dilations. 
It is easy to see that $$\overline{W(T)} \subseteq \cap \{\overline{W(U)}: U~\mbox{is a unitary dilation of}~ T\}.$$  Choi and Li showed that, in fact, 
$$\overline{W(T)} = \cap\{\overline{W(U)}: U \in \LL(\HH \oplus \HH)~\mbox{is a unitary dilation of}~ T\},$$ answering a question raised by Halmos (see, for example, \cite{Halmos}). We note that in the case that $\HH$ is $n$-dimensional, these unitary dilations are  $n$-dilations; that is, the dilations are of size $2n \times 2n$.

Before Choi and Li's work was completed, Gau and Wu \cite{GauWu} studied the so-called compressions of the shift on finite-dimensional spaces and their numerical ranges. If $\SS_n$ is the class of all completely nonunitary contractions $T$ (that is, $\|T\| \le 1$ and $T$ has no eigenvalue of modulus one) on an $n$-dimensional space with $\mbox{rank}(I - T^* T) = 1$, Gau and Wu \cite[Corollary 2.8]{GauWu} showed that, in fact, if $T \in \SS_n$, then
$$W(T) = \bigcap \{W(U): U ~\mbox{is an $(n + 1)$-dimensional unitary dilation of}~ T\}.$$ 
(There is no need  to take the  closure in the case of finite-dimensional spaces.)
Thus, the unitary dilations may be chosen to be $1$-dilations when $\mbox{rank}(I - T^* T)=1$.  An extension of this result can be found in \cite{GLW}: namely, if $T$ is an $n \times n$ contraction with $\mbox{rank}(I - T^* T) = k$, then
\begin{equation}\label{eq:economical}
W(T) = \bigcap \{W(U): U \in M_{n + k} ~\mbox{is a unitary $k$-dilation of}~ T\}.
\end{equation}

It is easy to see that if $\rank (I-T^*T)=k$, then $T$ has no unitary $\ell$-dilations for $\ell<k$, which explains why
 Gau, Li and Wu refer to~\eqref{eq:economical} in \cite{GLW} as the most ``economical'' solution to the Halmos problem. We also refer the reader to the papers \cite{GauWu1}, \cite{GauWu2}, \cite{GauWu3}, and \cite{Wu} for work related to this discussion. These authors, as well as others, (in particular, \cite{Mirman}, \cite{Mirman1}, and \cite{DGV}) have studied this problem from a geometric point of view.

The analogue of $\SS_n$ on a space of infinite dimension is the class of contractions with $\rank(I-T^*T)=\rank(I-TT^*)=1$ for which  $T^n$ and $T^*{}^n$ tend strongly to 0. It is well known (see, for instance,~\cite{SNF}) that such a $T$ is unitarily equivalent to some \emph{model operator} $S_\theta$ defined as follows: Suppose $S$ is the unilateral shift on $H^2$. For $\theta$  an inner function on the unit disc $\mathbb{D}$, define $K_\theta = H^2 \ominus \theta H^2$ and $S_\theta = P_{K_\theta} S|K_\theta$. (The operator  $S_\theta$ is often called a {\it compression} of the shift.) Noting that when $\theta(0) = 0$ all unitary $1$-dilations of $S_\theta$ are equivalent to rank-$1$ perturbations of $S_{z\theta}$, the authors of \cite{CGP}  show that when $\theta = B$ is a Blaschke product we have
$$\overline{W(S_B)} = \bigcap \{\overline{W(U)}: U~\mbox{a rank-1}\mbox{ perturbation of}~ S_{zB}\}.$$  

Our goal in this paper is to extend these results to operator-valued inner functions. After two preliminary sections, the main results appear in Section~\ref{se:main}, where we show that the closure of the numerical range of $S_{\Theta}$, where $\Theta$ is an inner function in $H^2(\mathbb{C}^N)$, is the intersection of the closures of the numerical ranges of an appropriate family of unitary dilations of $S_\Theta$  (see Corollary~\ref{co:after main}). In Theorem~\ref{th:C_0(N)} this result is extended to a larger class of contractions, called $C_0(N)$ (see Definition~\ref{de:C0(N)}).
In Section~\ref{N-dilations}, we describe the spectrum of the unitary dilations, obtaining a generalization of the scalar case. 
We conclude the paper with a brief discussion of a conjecture about the numerical ranges of contractions with finite defect index.

\section{Preliminaries}

\subsection{Matrix--valued analytic functions}

The basic reference that we will use for matrix--valued analytic functions (or, equivalently, functions with values in $\LL(\bbC^N)$) is~\cite{KK}; our definitions are simpler since we will consider only bounded (in the operator norm) analytic functions $F: \bbD \to \mathcal{M}_N$ (the set of $N \times N$ matrices). These share certain factorization properties similar to those of scalar analytic functions. 

A bounded analytic matrix-valued function $F: \bbD \to \mathcal{M}_N$ is called \emph{outer} if $\det F(z)$ is outer, and \emph{inner} if  the boundary values (which can be defined as radial limits almost everywhere) are isometries for almost all $e^{it}\in\bbT$. 

It is known \cite[Theorem 5.4]{KK} that any analytic bounded $F$ can be factorized as 
\begin{equation}\label{eq:i-o factorization}
F=\Theta E
\end{equation}
where $\Theta$ is inner and $E$ is outer, and, if $F=\hat \Theta \hat E$, then $\hat \Theta=\Theta V$, $\hat E=V^*E$ for some constant unitary $V$.

The inner function appearing in ~\eqref{eq:i-o factorization} can be further factorized in two parts. Recall that a Blaschke--Potapov factor $b(P,\lambda)(z)$  determined by a point $\lambda\in \bbD$ and an orthogonal projection $P$ on $\bbC^N$ is an inner function given by the formulas
\[
b(P,\lambda)(z)=\frac{|\lambda|}{\lambda} \frac{\lambda-z}{1-\bar\lambda z}P+(I-P) \text{ for }\lambda\not=0,\quad b(P,0)=zP+(I-P).
\]
A finite Blaschke--Potapov product is a product
\[
B_n(z)=b(P_1,\lambda_1)(z)\cdots b(P_n,\lambda_n)(z),
\]
for some $\lambda_j, P_j$, $j=1,\dots, n$. If $(\lambda_j)$ is a Blaschke sequence in $\bbD$ (that is, $\sum_j(1-|\lambda_j|)<\infty$), while $P_j$ is an arbitrary sequence of projections on $\bbC^N$, then the sequence $B_n(z)$ converges at each point $z\in \bbD$ to $B(z)$, where $B$ is an inner function denoted by ${\displaystyle\prod^\curvearrowright}_{j} b
(P_j,\lambda_j)$. A function that can be written as $B(z)V$, where $V$ is a constant unitary, is called an (infinite) Blaschke--Potapov product. The convergence is uniform on all compact subsets of $\bbD$. (Note that  such a function is sometimes called a \emph{left} Blaschke--Potapov product; we will not have the occasion to use right Blaschke--Potapov products.)
Finally, an inner function $\Theta$ is called  \emph{singular} if $\det \Theta(z)\not=0$ for all $z\in\bbD$.

With these definitions, Theorem~4.1 in~\cite{KK} states that any inner function $\Theta$ decomposes as $\Theta=BS$, where $B$ is a (finite or infinite) Blaschke--Potapov product and $S$ is singular. As in the case of inner--outer factorization, the decomposition is unique up to a unitary constant; more precisely, if we also have $\Theta=\hat B\hat S$ with $\hat B$ a Blaschke--Potapov product and $\hat S$ singular, then $\hat B=BV$ and $\hat S=V^*S$ for some constant unitary $V$.

The next lemma is a Frostman-type theorem that follows from~\cite{KK}.

\begin{lemma}\label{le:frostman}
Every inner function $\Theta$  in $H^2(\mathbb{C}^N)$ is a uniform limit of infinite Blaschke--Potapov products.
\end{lemma}

\begin{proof}
For $\lambda\in \bbD$,  $(\Theta - \lambda I)(I - \bar \lambda \Theta)^{-1}$ is inner and $I-\lambda\Theta$ is outer; thus 
\[
\Theta-\lambda I= \big((\Theta - \lambda I)(I - \bar \lambda \Theta)^{-1}\big) (I - \bar \lambda \Theta)
\] 
is the inner--outer factorization of $\Theta-\lambda I$. But Corollary~6.1 from~\cite{KK} says that for a dense set of $\lambda\in\bbD$ the inner factor of $\Theta-\lambda I$ is a Blaschke--Potapov product. If we take a sequence $\lambda_n\to 0$ with this property and we denote the corresponding Blaschke--Potapov product by $B^{(n)}$, then
\[
B^{(n)}=(\Theta- \lambda_n I)(I - \bar \lambda_n \Theta)^{-1},
\]
whence 
\[
\Theta=\lambda_n I+ B^{(n)}(I - \bar \lambda_n \Theta)= \lim_{n\to\infty}B^{(n)}.\qedhere
\]
\end{proof}

\subsection{Model spaces}

Let $\EE, \EE_*$ be Hilbert spaces. Suppose we are given an operator-valued inner function $\Theta(z):\EE\to \EE_*$. The \emph{model space} associated to it is 
\[
K_\Theta:= H^2(\EE_*)\ominus \Theta H^2(\EE),
\]
The operator $\mathbf T_\Theta :  H^2(\EE)\to  H^2(\EE_*)$ defined  by $\mathbf T_\Theta f=\Theta f$ is an isometry, and we have
 \begin{equation}\label{eq:projectiononK_Theta}
P_{K_\Theta}=I-\mathbf T_\Theta\mathbf T^*_\Theta.
\end{equation}

In particular, $\Theta(z)=zI_{\EE_*}:\EE_*\to \EE_*$ is inner; we will denote the corresponding $\mathbf{T}_\Theta$ simply by $\mathbf T_z$.
The \emph{model operator} $S_\Theta$ is the compression of $\mathbf T_z$ to $K_\Theta$; that is, $S_\Theta =P_{K_\Theta} \mathbf T_z P_{K_\Theta}| K_\Theta$.
 
 An inner function $\Theta$ is called \emph{pure} if it has no constant unitary direct summand; this is equivalent to assuming $\|\Theta(0)x\|<\|x\|$ for all $x\not=0$. A general inner function is the direct sum of a pure inner function and a unitary constant; from the point of view of model spaces and operators we may consider only pure inner functions. Thus, from now on, we assume that $\Theta$ is a pure inner function.

Recall that the defect operators and spaces of a contraction $T$ are defined by $D_T = (I - T^* T)^{1/2}$ and $\DD_T = \overline{\mbox{ran}\,D_T}$. The next lemma shows how one can identify the defect spaces of $S_\Theta$; a good reference is~\cite[Section~1]{F}.

\begin{lemma}\label{le:identification of defects} Suppose $\Theta(z):\EE\to \EE_*$ is a pure inner function; in particular, $D_{\Theta(0)}$ and $D_{\Theta(0)^*}$ have dense ranges.
Define the maps $\iota:\EE\to H^2(\EE_*)$, $\iota_*:\EE_*\to H^2(\EE_*)$ (on dense domains) by
\begin{equation}\label{eq:definitions of iota}
 \begin{split}
\iota (D_{\Theta(0)}\xi)& =\frac{1}{z}(\Theta(z)-\Theta(0))\xi,\quad \xi\in\EE;\\
\iota_*(D_{\Theta(0)^*}\xi_*)& =(I-\Theta(z)\Theta(0)^*)\xi_*, \quad \xi_*\in \EE_*.
\end{split}
\end{equation}
Then $\iota$ and $\iota_*$ are isometries with ranges $\DD_{S_\Theta}$ and $\DD_{S_\Theta^*}$ respectively, and the following diagram is commutative:
\begin{equation}\label{eq:commutative diagram}
 \begin{CD}
\EE @>\iota >> \DD_{S_\Theta}\\
@VV-\Theta(0)V  @VVS_\Theta V\\
\EE_* @>\iota_*>> \DD_{S_\Theta^*}
\end{CD}\quad.
\end{equation}
In particular, $\dim \DD_{S_\Theta}=\dim\EE$ and $\dim \DD_{S_\Theta^*}=\dim\EE_*$.
\end{lemma}
\noindent We will occasionally write $\iota^\Theta$ and $\iota_*^\Theta$ to indicate the dependence on $\Theta$.

From the  Sz-Nagy--Foias theory it follows that any $C_{.0}$ contraction $T$ (that is, a contraction such that the powers of the adjoint tend strongly to 0) is unitarily equivalent to some $S_\Theta$, where we can take $\EE=\DD_T$ and $\EE_*=\DD_{T^*}$. 

We are actually interested in the particular case when $\dim \DD_T=\dim \DD_{T^*}=N<\infty$. The following definition appears in~\cite{SNF}. 

\begin{definition}\label{de:C0(N)}
A contraction $T\in\LL(\HH)$ is said to be \emph{of class $C_0(N)$} if $\dim \DD_T=\dim \DD_{T^*}=N<\infty$ and $T^n, T^*{}^n$ tend strongly to~0.
\end{definition}

If $T\in C_0(N)$, then $T$ is unitarily equivalent to $S_\Theta$, with $\Theta(z):\bbC^N\to\bbC^N$.
In this case  $\Theta(0)$ is a strict contraction, and  formulas~\eqref{eq:definitions of iota} are defined on all of  $\bbC^N$. The next lemma collects a few facts that we shall use.

\begin{lemma}\label{le:blaschke-potapov}
Suppose $\Theta(z):\bbC^N\to\bbC^N$ is an inner function.

\begin{itemize}
\item[(i)] $K_\Theta$ is finite dimensional if and only if $\Theta$ is a finite Blaschke--Potapov product.

\item[(ii)] If $\Theta_n\to \Theta$ in the uniform norm, then $P_{K_{\Theta_n}}\to P_{K_\Theta}$ uniformly.

\item[(iii)] Let $b_j = b(P_j, \lambda_j)$. If $B=\displaystyle\prod^\curvearrowright _k b_k$ is an infinite Blaschke--Potapov product and $B_n=b_1\cdots b_n$, then
\begin{itemize}
\item[(a)]
 $K_B=\overline{\bigcup_n K_{B_n}}$;
 
 \item[(b)]
$B_n \xi\to B\xi$ in $H^2(\bbC^N)$, for any $\xi\in \bbC^N$.

\end{itemize}

\end{itemize}
\end{lemma}

\begin{proof} Statement (i) can be found, for instance, in~\cite[Ch.2, Lemma 5.1]{P}, while (ii) follows from~\eqref{eq:projectiononK_Theta}.

As for (iii), a standard normal family argument shows that  $BH^2(\bbC^N)=\bigcap_n B_n H^2(\bbC^N)$, and therefore  (a) follows by passing to orthogonal complements. 

For (b), write $B=B_n\tilde{B}_n$, where $\tilde{B}_n$ is also an infinite Blaschke--Potapov product. If $B(0)$ is invertible, the pointwise convergence of $B_n$ to $B$ implies that  $\tilde{B}_n(0)\to I_{\bbC^N}$, whence (taking norms and scalar products in $H^2(\bbC^N)$)  
\[
\begin{split}
\|B_n\xi-B\xi\|^2&=2\|\xi\|^2-2\Re\<B_n\xi,B\xi\>=
 2\|\xi\|^2-2\Re\<\xi,\tilde{B}_n\xi\>\\&=
 2\|\xi\|^2-2\Re\<\xi,\tilde{B}_n(0)\xi\>\to 0.
\end{split}
\]
In the general case,  write $B_n = C D_n$, where $C$ contains the Blaschke--Potapov factors $b(P,\lambda)$ corresponding to $\lambda=0$. We have then $B=CD$ (with $D$ an infinite Blaschke--Potapov product), while the previous argument shows that $D_n\xi\to D\xi$ in $H^2(\bbC^N)$. Multiplying with the inner function $C$ yields the result.
\end{proof}

%

\section{Unitary $N$-dilations}

The next result is folklore; we give a short proof.

\begin{proposition}\label{pr:folklore}
 Suppose $T\in\LL(\HH)$ is a contraction such that $\dim \DD_T= \dim\DD_{T^*}=N<\infty$. If $U\in\LL(\HH\oplus\EE)$ is a unitary $N$-dilation of $T$, then there exist unitary operators $\omega:\EE\to\DD_T, \omega_*:\EE\to\DD_{T^*}$, such that
\begin{equation}\label{eq:general form of d-dilations}
 U=\begin{pmatrix}
    T& D_{T^*}\omega\\
\omega_*^* D_T& -\omega_*^* T^* \omega
   \end{pmatrix}.
\end{equation}
Conversely, any choice of $\omega:\EE\to\DD_T, \omega_*:\EE\to\DD_{T^*}$ yields, through formula~\eqref{eq:general form of d-dilations}, a unitary $N$-dilation of $T$.
\end{proposition}

\begin{proof}
 Theorem~1.3 of \cite{AG} says that if
\[
 \begin{pmatrix}
    T& T_{12}\\
T_{21}& T_{22}
   \end{pmatrix}: \HH\oplus\EE \to\HH\oplus\EE 
\]
is a contraction, then there exist contractions $\Gamma_1:\EE\to \DD_{T^*}$, $\Gamma_2:\DD_T\to\EE$ and $\Gamma:\DD_{\Gamma_1}\to\DD_{\Gamma_2^*}$ such that $T_{12}=D_{T^*}\Gamma_1$, $T_{21}=\Gamma_2 D_T$ and 
$T_{22}=- \Gamma_2 T^*\Gamma_1+D_{\Gamma_2^*}\Gamma D_{\Gamma_1}$.  We apply this result to $U$. 

Since 
\[
\| U(x\oplus 0) \|^2 =\|Tx\|^2+\|\Gamma_2 D_Tx\|^2
\le \|Tx\|^2+\|D_Tx\|^2 = \|x\|^2,
\]
and the first column of $U$ is an isometry, the last term is equal to the first; so the middle inequality is an
equality. This means that $\Gamma_2$ acts isometrically on the image
of $D_T$; but this is precisely $\mathcal{D}_T$, whence $\Gamma_2$ has to be an isometry. 

In fact, $\Gamma_2$ is  unitary, since it acts between spaces of the same dimension~$N$. Similarly, we obtain that $\Gamma_1$ is  unitary, which implies $\Gamma$ acts between~$0$ spaces. The result follows if we let  $\omega=\Gamma_1$ and $\omega_*=\Gamma_2^*$.

The converse is immediate.
\end{proof}

We can write~\eqref{eq:general form of d-dilations} as
\begin{equation*}
 U=\begin{pmatrix}
    I&0\\ 0& \omega_*^*
   \end{pmatrix}
\begin{pmatrix}
    T& D_{T^*}\\
 D_T& - T^* 
   \end{pmatrix}
\begin{pmatrix}
    I&0\\ 0& \omega
   \end{pmatrix}.
\end{equation*}
The next corollary follows immediately from this formula.

\begin{corollary}\label{co:d-dilations}
Suppose $T\in\LL(\HH)$ is a contraction with $\dim \DD_T= \dim\DD_{T^*}=N<\infty$ and
\[
\mathbf{u}=\begin{pmatrix}
T&\mathbf u_{12}\\\mathbf u_{21}&\mathbf u_{22}
\end{pmatrix}\in\LL(\HH\oplus \EE)
\]
is a unitary $N$-dilation of $T$. Then any unitary $N$-dilation $U$ of $T$ on $\HH\oplus \EE'$  (for some Hilbert space $\EE'$ with $\dim\EE'=N$) is given by the formula
\begin{equation}\label{eq:formula with three terms}
 U=\begin{pmatrix}
    I&0\\ 0& \Omega_*^*
   \end{pmatrix}
\mathbf u
\begin{pmatrix}
    I&0\\ 0& \Omega
   \end{pmatrix}
\end{equation}
where $\Omega,\Omega_*:\EE'\to\EE$ are unitaries.
\end{corollary}

We are interested in consequences for $S_\Theta$.  The notation below refers to that of Lemma~\ref{le:identification of defects}.

\begin{lemma}\label{le:formulas for extensions}
All unitary $N$-dilations of $S_\Theta$ to $K_\Theta\oplus \bbC^N$ can be indexed by unitaries $\Omega,\Omega_*:\bbC^N\to\bbC^N$, according to the formula
\begin{equation}\label{eq:UOmegaOmega*}
U_{\Omega,\Omega_*}=
\begin{pmatrix}
S_\Theta & \iota_* D_{\Theta(0)^*}\Omega\\
\Omega_*^* D_{\Theta(0)} \iota^* & \Omega_*^*\Theta(0)^*\Omega
\end{pmatrix}.
\end{equation}
\end{lemma}

\begin{proof} Let us apply Proposition~\ref{pr:folklore} (the part stated as a converse) to the case $T=S_\Theta$, $\EE=\bbC^N$, $\omega=\iota_*$, $\omega_*=\iota$. 
We obtain the following $N$-dilation of $T$ to $\HH\oplus \bbC^N$:
\begin{equation*}
V=\begin{pmatrix}
    S_\Theta& D_{S_\Theta^*}\iota_*\\
\iota^* D_{S_\Theta}& -\iota^* S_\Theta^* \iota_*
   \end{pmatrix}.
\end{equation*}

The commutative diagram~\eqref{eq:commutative diagram} yields the relations
\begin{equation}\label{eq:intertwining}
 S_\Theta\iota=-\iota_* \Theta(0), \qquad \iota^* S_\Theta^*= -\Theta(0)^* \iota_*^*.
\end{equation}
It follows immediately that $-\iota^* S_\Theta^* \iota_*=\Theta(0)^*$. 

From~\eqref{eq:intertwining} we have $\iota^*S_\Theta^* S_\Theta \iota=\Theta(0)^*\Theta(0)$, whence 
$\iota^*(I_\HH-S_\Theta^* S_\Theta) \iota=I_{\bbC^N}-\Theta(0)^*\Theta(0)$ and thus 
$\iota^*D_{S_\Theta}\iota =D_{\Theta(0)}$. Multiplying the last relation with $\iota^*$ on the right, and taking into account that $\iota\iota^*=P_{\DD_{S_\Theta}}$ and $D_{S_\Theta}P_{\DD_{S_\Theta}}=D_{S_\Theta}$, we obtain
\[
 \iota^*D_{S_\Theta}=D_{\Theta(0)}\iota^*.
\]
A similar computation yields
\[
 D_{S_\Theta^*}\iota_* =\iota_* D_{\Theta(0)^*},
\]
and thus
\[
 V=\begin{pmatrix}
    S_\Theta&\iota_* D_{\Theta(0)^*} \\
D_{\Theta(0)}\iota^*& \Theta(0)^*
   \end{pmatrix}.
\]
Applying Corollary~\ref{co:d-dilations} to $\EE=\EE'=\bbC^N$ and $\mathbf u=V$ finishes the proof.
\end{proof}

We have thus parametrized all unitary $N$-dilations of $S_\Theta$ to $K_\Theta\oplus \bbC^N$ by pairs of unitaries on $\bbC^N$. If we are interested only in classes of unitary equivalence, we may take a single unitary as parameter, since $U_{\Omega,\Omega_*}$, $U_{\Omega\Omega_*^*,I}$, and $U_{I,\Omega_*\Omega^*}$ are all unitarily equivalent, and therefore  have the same numerical range. In the sequel, we let
\begin{equation}\label{eq:definition of U_Omega}
 U_{\Omega}^\Theta:=
\begin{pmatrix}
S_\Theta & \iota_* D_{\Theta(0)^*}\Omega\\
D_{\Theta(0)} \iota^* & \Theta(0)^*\Omega
\end{pmatrix}.
\end{equation}

%

\section{The main result}\label{se:main}

Let $\KKK$ denote the complete metric space of all nonempty compact subsets of $\bbC$, endowed with the Hausdorff distance $\ddd$. Suppose that $A\in\KKK$ and $\tau:X\to\KKK$ is a continuous mapping defined on some compact space $X$. We will say that \emph{$\tau$ wraps $A$} if for each open half-plane $\bbH$ in $\bbC$ that contains $A$ there exists $x\in X$ such that $\tau(x)\subset \bbH$. 

\begin{lemma}\label{le:hausdorff}
Let $A_n,A\in \KKK$ with $A_n\to A$.  Let $X$ be a compact space, and $\tau_n,\tau:X\to \KKK$ be continuous mappings such that $\tau_n\to \tau$ uniformly on $X$. Suppose that for each $n$, $\tau_n$ wraps $A_n$. Then $\tau$ wraps $A$.
\end{lemma}

\begin{proof} If $\bbH$ is an open half-plane and $A\subset\bbH$, let $\bbH'$ be a slight translate of $\bbH$ towards $A$ such that we still have $A\subset \bbH'$.  
 For $n$ sufficiently large $A_n\subset \bbH'$. It  follows then from the assumption that for each $n$ sufficiently large there exists $x_n\in  X$ such that $\tau_n(x_n)\subset \bbH'$. Letting $x$ be a limit point of $x_n$ in $X$, a simple $\epsilon/2$  argument shows that $\tau(x)\subset \bbH$.  
\end{proof}
 
\begin{remark}\label{re:wrapping and intersection}
Suppose $A\subset\tau(x)$ for all $x$, and $A$ and $\tau(x)$ are convex for all $x$. If $\tau$ wraps $A$, then $A=\cap_{x\in X}\tau(x)$. The converse is not true, as can easily be seen by considering $A$ to be the intersection of two line segments. However, the result that we quote below (in Theorem~\ref{th:main}, Step 1) from~\cite{GLW} actually yields a wrapping property of $A$, not only intersection.
\end{remark}

The following simple lemma will be used in Section~\ref{se:final}.

\begin{lemma}\label{le:adding a set to wrapping}
Suppose $A\in\KKK$ and $\tau:X\to\KKK$ wraps $A$. If $B\in\KKK$, $\tilde A:=\conv(A,B)$, $\tilde \tau(x)=\conv(\tau(x),B)$, then $\tilde \tau$ wraps $\tilde A$.
\end{lemma}

\begin{proof}
Take a half-plane $\bbH$ that contains $\tilde A$. Then it contains $A$ and $B$. By hypothesis, there exists $x\in X$ such that $\tau(x)\subset\bbH$. Since $\bbH$ is convex, it follows that $\tilde \tau(x)\subset \bbH$, which proves the lemma.
\end{proof}

The elements of $\KKK$ that we will consider are closures of numerical ranges. 
The next lemma states some continuity properties for these sets.

 \begin{lemma}\label{le:cont}
 (i) Let $T, S \in \LL(H)$. Then $\ddd(\overline{W(T)},\overline{W(S)})\le \|T-S\|$.

 (ii) If $H_n\subset H_{n+1}\subset\cdots\subset H$ and $\overline{\bigcup_{n}H_n}=H$, then for all $T\in\LL(H)$,
 \[
 \overline{W(T)}=\overline{\bigcup_n {W(P_{H_n}TP_{H_n}|H_n)}}.
 \]
In particular, $\ddd(\overline{W(P_{H_n}TP_{H_n}|H_n)}, \overline{W(T)})\to 0$.

 (iii) Suppose $T\in\LL(H)$, and $P,Q$ are orthogonal projections on $H$, with $\|P-Q\|<1$. Then 
 \[\mathfrak{d}(\overline{W(PTP|PH)}, \overline{W(QTQ|QH)})\le \|T\|\cdot \|P-Q\|
 \left[ 1+\frac{2}{(1-\|P-Q\|)^2} \right]
 .\] 
 In particular, if $P_n,P$ are orthogonal projections and $P_n\to P$ uniformly, then $\ddd(\overline{W(P_{n}TP_{n}|P_nH)}, \overline{W(PTP|PH)})\to 0$.
 \end{lemma}

In the sequel we will let $X$ denote the space of unitary operators on $\bbC^N$ and we define $\tau^\Theta(\Omega)=\overline{W(U^\Theta_{\Omega})}$, where $U^\Theta_{\Omega}$ is given by~\eqref{eq:definition of U_Omega}.

The next lemma singles out a technical argument that will be used twice in the proof of Theorem~\ref{th:main}.

\begin{lemma}\label{le:main argument}
Suppose that $\Theta, \Theta_n:\bbD\to \LL(\bbC^N)$ are inner functions, such that 
\begin{itemize}

\item[(a)]
$\Theta_n\xi\to \Theta\xi$ in $H^2(\bbC^N)$, for any $\xi\in \bbC^N$;

\item[(b)] $\ddd(\overline{W(S_{\Theta_n})}, \overline{W(S_{\Theta})})\to 0$;

\item[(c)]
if  we define
 \begin{equation*}\label{eq:formula for V}
 V_{n,\Omega} = 
 \begin{pmatrix}
 S_{\Theta_n} & P_{K_{\Theta_n}}\iota_* D_{\Theta(0)^*}\Omega\\
 \noalign{\medskip}
D_{\Theta(0)} \iota^*P_{K_{\Theta_n}} & \Theta(0)^*\Omega
 \end{pmatrix},
 \end{equation*}
then $\ddd(\overline{W(V_{n,\Omega})}, \overline{W(U^\Theta_\Omega)})\to 0$ uniformly in $\Omega$.
\end{itemize}

If $\tau^{\Theta_n}$ wraps $\overline{W(S_{\Theta_n})}$ for all $n$, then $\tau^{\Theta}$ wraps $\overline{W(S_{\Theta})}$.
\end{lemma}

\begin{proof}
Condition (a) implies, by formulas~\eqref{eq:definitions of iota}, that $\iota^{\Theta_n}\to \iota^\Theta$ and $\iota_*^{\Theta_n}\to \iota_*^\Theta$; whence $\|V_{n,\Omega}-U^{{\Theta_n}}_\Omega\|\to 0$ uniformly in $\Omega$. Therefore $\ddd(\overline{W( V_{n,\Omega}	)},  \overline{W( U_{\Omega}^{\Theta_n})})\to 0$ uniformly in $\Omega$, which, together with (c), yields $\ddd(\overline{W(U_{\Omega}^\Theta	)},  \overline{W( U_{\Omega}^{\Theta_n})})\to 0$ uniformly in $\Omega$.

We may then apply Lemma~\ref{le:hausdorff} with $A_n=\overline{W(S_{\Theta_n})}$, $A=\overline{W(S_{\Theta})}$,   $\tau_n=\tau^{\Theta_n}$, and $\tau=\tau^\Theta$. With these notations, the assumption of Lemma~\ref{le:main argument} becomes that $\tau_n$ wraps $A_n$, and it follows that $\tau$ wraps $A$.
\end{proof}

 \begin{theorem}\label{th:main}   For any inner function $\Theta: \bbD\to \LL(\bbC^N)$, the map $\tau^\Theta$ wraps $\overline{W(S_\Theta)}$.
 \end{theorem}

 \begin{proof} The proof will be done in three steps.
 
 \smallskip
{\bf Step 1.} In case $\Theta$ is a finite Blaschke--Potapov product, the space $K_\Theta$ is finite dimensional and the statement is a consequence of~\cite[Theorem 1.2]{GLW} (see Remark~\ref{re:wrapping and intersection}).

\smallskip
{\bf Step 2.} To pass to infinite Blaschke--Potapov products, suppose that $\Theta=B$ and $\Theta_n= B_n$ (where the notation is as in Lemma~\ref{le:blaschke-potapov} (iii)). We want to use Lemma~\ref{le:main argument}. Condition (a) therein is satisfied by Lemma~\ref{le:blaschke-potapov} (iii)(b).
Applying Lemma~\ref{le:cont} (ii) to $T=S_B$, $H=K_B$, $H_n=K_{B_n}$, we obtain $\ddd(\overline{W(S_{\Theta_n})}, \overline{W(S_{\Theta})})\to 0$,
 and therefore (b) is also satisfied. Finally, to obtain (c), we apply Lemma~\ref{le:cont} (ii) again, this time  to $T=U_{\Omega}^B$, $H=K_B\oplus \bbC^N$, $H_n=K_{B_n}\oplus \bbC^N$. By Step~1 we know that $\tau^{B_n}$ wraps $\overline{W(S_{B_n})}$ for all $n$, and Lemma~\ref{le:main argument} implies that $\tau^{B}$ wraps $\overline{W(S_{B})}$.

 \smallskip
{\bf Step 3.}  According to Lemma~\ref{le:frostman}, we take a sequence of Blaschke--Potapov products $\Theta_n$ that tend uniformly to an arbitrary inner function $\Theta$. 
Condition (a) in Lemma~\ref{le:main argument} is obviously satisfied.
By Lemma~\ref{le:blaschke-potapov} (ii), we have $P_{K_{\Theta_n}}\to P_{K_\Theta}$ uniformly.  Since $S_{\Theta_n}=P_{K_{\Theta_n}} \mathbf T_z  P_{K_{\Theta_n}}|K_{\Theta_n}$ and $S_\Theta= P_{K_{\Theta}} \mathbf T_z  P_{K_{\Theta}}|K_{\Theta}$, Lemma~\ref{le:cont} (iii), applied to $H=H^2(\bbC^N)$, $T=\mathbf T_z$, $P_n=P_{K_{\Theta_n}}$, and $P=P_{K_{\Theta}}$,  yields condition (b) in Lemma~\ref{le:main argument}.

To obtain (c), apply Lemma~\ref{le:cont} (iii) again, this time to $H=H^2(\bbC^N)\oplus \bbC^N$, $P_n=P_{K_{\Theta_n}\oplus \bbC^N}$, $P=P_{K_\Theta\oplus \bbC^N}$, and
\[
T=\begin{pmatrix}
T_z & \iota_* D_{\Theta(0)^*}\Omega\\
D_{\Theta(0)}\iota^*& \Theta(0)^*\Omega
\end{pmatrix}.
\]
Once again, we use the fact that $P_{K_{\Theta_n}}\to P_{K_\Theta}$ uniformly to conclude that (c) is also satisfied.

By Step~2 we know that $\tau^{\Theta_n}$ wraps $\overline{W(S_{\Theta_n})}$ for all $n$, and Lemma~\ref{le:main argument} implies that $\tau^{\Theta}$ wraps $\overline{W(S_{\Theta})}$. The proof of the theorem is finished.
 \end{proof}

The next corollary is a consequence of Remark~\ref{re:wrapping and intersection}.

\begin{corollary}\label{co:after main}
Suppose  $\Theta: \bbD\to \LL(\bbC^N)$ is an inner function. Then 
 \begin{equation}\label{eq:main}
 \overline{W(S_\Theta)}= \bigcap_{\Omega} \overline{W(U^\Theta_{\Omega})},
\end{equation}
where $U^\Theta_{\Omega}$ is defined by~\eqref{eq:definition of U_Omega}, while the intersection is taken with respect to all unitary operators $\Omega$ on $\bbC^N$.
\end{corollary}

Since  $C_0(N)$ contractions are unitarily equivalent to model operators $S_\Theta$, with $\Theta: \bbD\to \LL(\bbC^N)$  inner, we 
may extend the result to this class.

\begin{theorem}\label{th:C_0(N)}
Suppose $T\in \LL(\HH)$ is a contraction of  class $C_0(N)$, $\UU$ is the set of unitary $N$-dilations of $T$ to $\HH\oplus \bbC^N$, and $\tau:\UU\to\KKK$ is defined by $\tau(\U)=\overline{W(\U)}$. Then $\tau$ wraps $\overline{W(T)}$. In particular,
\begin{equation*}
\overline{W(T)}=\bigcap_{\U\in\UU}\overline{W(\U)}.
\end{equation*}
\end{theorem} 

\section{Spectrum and numerical range of $N$-dilations}
\label{N-dilations}

In the case that $\Theta$ is a finite (scalar) Blaschke product,
the spectrum of the extensions $U^\Theta_{\Omega}$ can be identified precisely. Since $U^\Theta_{\Omega}$ is a unitary operator, the numerical range is the (closed) convex hull of the spectrum, and we obtain a complete description of $W(U^\Theta_{\Omega})$,  (see, for example, \cite{CGP},  \cite{DGV}, and \cite{GauWu}). The same can be  done in the case of a general matrix-valued inner function, by relating these functions to perturbations of a ``slightly larger'' model operator. We need some preliminary material, for which the reference is~\cite{F}.

If $T\in\LL(\HH)$ is a contraction, then $T(\DD_T)\subset \DD_{T^*}$, $T(\DD_T^\perp)\subset \DD_{T^*}^\perp$, and $T$ acts unitarily from $\DD_T^\perp$ onto $\DD_{T^*}^\perp$. For $A:\DD_T\to \DD_{T^*}$, define $T[A]\in \LL(\HH)$ by the formula
\begin{equation}\label{eq:definition of T[A]}
T[A]x=\begin{cases}
Ax& \text{if }x\in \DD_T,\\
Tx& \text{if }x\in \DD_T^\perp.
\end{cases}
\end{equation}
It is easy to see that $T[A]$ is a contraction (respectively isometry, coisometry, unitary) if and only if  $A$ is a contraction (respectively isometry, coisometry, unitary). 

We will be interested in the particular situation when $T=S_\Xi$, with $\Xi(z)=z\Theta(z)$, with $\Theta: \bbD\to \LL(\bbC^N)$  an inner function and  $A$ unitary. 
According to formulas~\eqref{eq:definitions of iota}, we have then $\DD_{S_\Xi}=\Theta\bbC^N$, $\DD_{S_\Xi^*}=\bbC^N$. Thus a unitary mapping $A:\DD_T\to \DD_{T^*}$ is given by $A\Theta(z)\xi=\omega\xi$, with $\omega:\bbC^N\to\bbC^N$ unitary. We will write $A=A_\omega$. Since
\[
K_\Xi=K_\Theta\oplus \Theta\bbC^N=zK_\Theta\oplus \bbC^N,
\]
we have unitary operators $J,J_*:K_\Theta\oplus \bbC^N\to K_\Xi$ defined by
\[
J(f\oplus \xi)=f+ \Theta \xi,\qquad J_*(f\oplus \xi)=zf+ \xi.
\]

We will write 
\[
Z_\Xi(\omega):=J^*S_\Xi[A_\omega]J\in\LL(K_\Theta\oplus\bbC^N).
\]
With these notations,  the  lemma below follows  from \cite[Theorem 3.6]{F}. For a more general result, see  \cite[Theorem 4.5]{BL}.

\begin{lemma}\label{le:spectrum of Z[omega]}
With the above assumptions, the spectrum of $Z_\Xi(\omega)$ is the union of the sets of points $\zeta\in\bbT$ at which $\Xi$ has no analytic continuation and the set of points $\zeta\in\bbT$ at which $\Xi$ has an analytic continuation but $\Xi(\zeta)-\omega$ is not invertible.

In particular, if $\Xi$ is a finite Blaschke--Potapov product, then \[
\sigma(Z_\Xi(\omega))=\{\zeta\in\bbT: \det(\Xi(\zeta)-\omega)=0\}.
\]
\end{lemma}

The relation with $N$-dilations is given by the next proposition.

\begin{proposition}\label{pr:augment to unitary perturbations}
Suppose $\Theta: \bbD\to \LL(\bbC^N)$ is an inner function. Define $\Xi(z)=z\Theta(z)$. Then $U_{\Omega}^\Theta=Z_{\Xi}(\Omega)$. 
\end{proposition}

\begin{proof}
We have
\begin{equation}\label{eq:formulaZ1}
Z_\Xi(\Omega)=J^*S_\Xi[A_{\Omega}]J=(J^*J_*)(J_*^*  S_\Xi[A_{\Omega}]J) .
\end{equation}
Since
\[
S_\Xi[A_{\Omega}]J(f\oplus \xi)=
S_\Xi[A_{\Omega}](f+\Theta\xi)
=zf+\Omega\xi=J_*(f\oplus \Omega\xi),
\]
it follows that
\begin{equation}\label{eq:formulaZ2}
(J_*^*  S_\Xi[A_{\Omega}]J) (f\oplus \xi)= f\oplus \Omega\xi.
\end{equation}

To compute $J^* J_*:K_\Theta\oplus \bbC^N\to K_\Theta\oplus \bbC^N$, denote the corresponding matrix by $\begin{pmatrix}
A_{11}&A_{12}\\ A_{21}& A_{22}
\end{pmatrix}$, and let $P_i$ denote the projection on the $i$-th component in $K_\Theta\oplus \bbC^N$.  We have 
\[
A_{11}(f)=P_1(J^* J_*(f\oplus 0))
=P_1 (J^* zf)=P_{K_\Theta}zf=S_\Theta f.
\]
Further, $J_*(0\oplus\xi)=\xi$, viewed as a constant function in $K_\Xi$. This decomposes with respect to $K_\Xi=K_\Theta\oplus \Theta\bbC^N$ as
\[
\xi=(1-\Theta\Theta(0)^*)\xi + \Theta\Theta(0)^*\xi 
=\iota_*D_{\Theta(0)^*}\xi + \Theta\Theta(0)^*\xi.
\]
It follows that 
\[
J^*J_*(0\oplus\xi)= \iota_*D_{\Theta(0)^*}\xi\oplus \Theta(0)^*\xi,
\]
and thus 
\[
A_{12}=\iota_*D_{\Theta(0)^*}, \qquad A_{22}= \Theta(0)^*.
\]
To obtain $A_{21}$, we work now with the adjoint map $J_*^*J$. We have $J(0\oplus \xi)=\Theta\xi$, and the last function decomposes with respect to $K_\Xi=zK_\Theta\oplus \bbC^N$ as
\[
\Theta\xi=z\left( \frac{\Theta-\Theta(0)}{z} \right)\xi+\Theta(0)\xi
=z\iota D_{\Theta(0)}\xi +\Theta(0)\xi,
\]
whence
\[
J_*^*J (0\oplus \xi) =\iota D_{\Theta(0)}\xi\oplus \Theta(0)\xi.
\]
Therefore
\[
A_{21}^*=\iota D_{\Theta(0)}, \qquad A_{21}=D_{\Theta(0)}\iota^*.
\]
Finally,
\begin{equation}\label{eq:formulaZ3}
J^*J=
\begin{pmatrix}
S_\Theta & \iota_*D_{\Theta(0)^*}\\
D_{\Theta(0)}\iota^* & \Theta(0)^*
\end{pmatrix}
\end{equation}
Now the proof follows by comparing equations~\eqref{eq:formulaZ1},~\eqref{eq:formulaZ2}, and~\eqref{eq:formulaZ3} with~\eqref{eq:definition of U_Omega}.
\end{proof}

From Lemma~\ref{le:spectrum of Z[omega]} and Proposition~\ref{pr:augment to unitary perturbations},
the final result about spectrum and numerical range of $N$-dilations follows.

\begin{theorem}\label{th:numerical range of d-dilations}
With the above notations, the spectrum 
 $\sigma(U^\Theta_{\Omega})$ is the union of the sets of points $\zeta\in\bbT$ at which $\Theta$ has no analytic continuation and the set of points $\zeta\in\bbT$ at which $\Theta$ has an analytic continuation but $\zeta\Theta(\zeta)-\Omega$ is not invertible, while $\overline{W(U^\Theta_{\Omega})}$ is the closed convex hull of  $\sigma(U^\Theta_{\Omega})$.

In particular, if $\Theta$ is a finite Blaschke--Potapov product, then \[
\sigma(U^\Theta_{\Omega})=\{\zeta\in\bbT: \det(\zeta\Theta(\zeta)-\Omega)=0\}
\]
and $\overline{W(U^\Theta_{\Omega})}$ is the closed convex hull of the zeros of the polynomial $\det(\zeta\Theta(\zeta)-\Omega)$.
\end{theorem}

The scalar case of Theorem~\ref{th:numerical range of d-dilations} is contained in~\cite[Theorem 6.3]{GauWu2}.

\section{Final remarks}\label{se:final}

It seems natural to formulate the following conjecture, which would complement Choi and Li's answer to Halmos' question.

\begin{conjecture}\label{conjecture}
Suppose $T\in \LL(\HH)$ is a contraction with $\dim\DD_T=\dim \DD_{T^*}=N<\infty$, $\UU$ is the set of unitary $N$-dilations of $T$ to $\HH\oplus \bbC^N$, and $\tau:\UU\to\KKK$ is defined by $\tau(\U)=\overline{W(\U)}$. Then $\tau$ wraps $\overline{W(T)}$. In particular,
\begin{equation}\label{eq:general conjecture}
\overline{W(T)}=\bigcap_{\U\in\UU}\overline{W(\U)}
\end{equation}
\end{conjecture}

Note that the conjecture is open even for $N=1$. The main points that have been settled are presented below. In the sequel $T\in \LL(\HH)$ will be a contraction with $\dim\DD_T=\dim \DD_{T^*}=N<\infty$.


\medskip
\noindent {\bf 6.1.\ }
Theorem~\ref{th:C_0(N)} shows that the conjecture is true for $C_0(N)$ contractions.


\medskip
\noindent {\bf 6.2.\ }
As we show below, if we add a unitary operator to one for which the conjecture holds, the conjecture will still hold.

\begin{lemma}\label{le:adding a unitary}
If $T_1=T\oplus V$, where $V$ is unitary, and Conjecture~\ref{conjecture} is true for $T$, then it is true for $T_1$.
\end{lemma}

\begin{proof}
The unitary $N$-dilations of $T\oplus V$, for $V$ unitary, are exactly $\U\oplus V$, with $\U$ a unitary $N$-dilation of $T$. Since the numerical range of a direct sum is the convex hull of the numerical ranges of the components, the statement follows from Lemma~\ref{le:adding a set to wrapping}.
\end{proof}

Again by~\cite{SNF}, it is known that an arbitrary contraction is the direct sum of a completely  nonunitary contraction and a unitary; it follows then from Lemma~\ref{le:adding a unitary} that it is enough to prove Conjecture~\ref{conjecture} for a completely nonunitary~$T$.


\medskip
\noindent {\bf 6.3.\ }
We now specialize to the case $N=1$. Suppose $T$ is a completely nonunitary contraction with scalar characteristic function $\theta$~\cite{SNF}; now $\theta$ is an arbitrary function in the unit ball of $H^\infty$. Then $T$ is unitarily equivalent to the model operator $\T_\theta\in \LL(\K_\theta)$, where
\[
\K_\theta=(H^2\oplus L^2(\Delta))\ominus \{\theta f\oplus (1-|\theta|^2)^{1/2}f: f\in H^2\},
\]
with $\Delta=\{\zeta\in\bbT: |\theta(\zeta)|<1\}$, while 
$
\T_\theta(f\oplus g)=P_{\K_\theta} (zf\oplus \zeta g)
$.
 
If $\theta$ is inner, then $\Delta=\emptyset$ and we are back in the $C_{0}(1)$ case discussed in~{6.1}. 

On the other hand, the spectrum of $\T_\theta$ may be precisely identified in terms of the characteristic function: $\sigma(\T_\theta)$ is the union of the zeros of $\theta$ inside $\bbD$ and the complement of the open arcs of $\bbT$ on which $|\theta(\zeta)|=1$ and through which $\theta$ has an analytic extension outside the unit disk (see again~\cite{SNF} for a general statement; in the scalar case it was known earlier and is usually called the Livsic--Moeller theorem). 

In particular, it follows that Conjecture~\ref{conjecture} can be settled for a situation at the opposite extreme of the case in which $\theta$ is inner. Namely, if $|\theta(\zeta)|<1$ almost everywhere on $\bbT$, then $\sigma(\T_\theta)\supset \bbT$. In this case, Conjecture~\ref{conjecture} is trivially true: $\overline{W(T)}$ as well as  every $\overline{W(U)}$ must  equal~$\overline{\bbD}$. 

\bigskip
A final remark: the case $\dim\DD_T=\dim \DD_{T^*}=N<\infty$ is the only one in which we can hope to obtain the numerical range of $T$ by using ``economical'' unitary dilations. If $\dim\DD_T\not=\dim \DD_{T^*}$, or if both dimensions are infinite, then it is easy to see that for any unitary dilation $U\in\LL(\KK)$ of $T\in\LL(\HH)$ one must have $\dim(\KK\ominus\HH)=\infty$.

\section*{Acknowledgements}

We thank the referee for useful remarks and, in particular, for pointing  
out the need for a slight change to the proof of Theorem~\ref{th:main}.

\end{document}